\documentclass[secnum,leqno]{rims-bessatsufa}
\usepackage{amssymb, amsmath}
\usepackage{graphics}
\usepackage[dvips]{graphicx}
\usepackage{eepic}
\usepackage{epic}

\newtheorem{thm}{Theorem}[section]
\newtheorem{crl}[thm]{Corollary}

\newtheorem{lmm}[thm]{Lemma}

\theoremstyle{definition}

\theoremstyle{remark}

\title{Quenched localisation in the Bouchaud trap model with regularly varying traps}
\author{David \textsc{Croydon}\footnote{Department of Statistics, University of Warwick, Coventry, CV4 7AL, United Kingdom.\newline e-mail: \texttt{d.a.croydon@warwick.ac.uk}}
          ~and Stephen \textsc{Muirhead}\footnote{Mathematical Institute
University of Oxford, Andrew Wiles Building, Radcliffe Observatory Quarter, Oxford, OX2 6GG, United Kingdom. \endgraf e-mail: \texttt{muirhead@maths.ox.ac.uk}}}
\AuthorHead{David Croydon and Stephen Muirhead}         
\support{The second author was supported by the Engineering \& Physical Sciences Research Council (EPSRC) Fellowship EP/M002896/1 held by Dmitry Belyaev.}
\VolumeNo{x}           
\YearNo{201x}           
\PagesNo{000--000}      
\begin{document}

\maketitle

\begin{abstract}
This article describes the quenched localisation behaviour of the Bouchaud trap model on the integers with regularly varying traps. In particular, it establishes that for almost every trapping landscape there exist arbitrarily large times at which the system is highly localised on one site, and also arbitrarily large times at which the system is completely delocalised.
\end{abstract}

\section{Introduction}

The Bouchaud trap model (BTM) was introduced in \cite{Bou92} as a simple way of investigating the evolution of a physical system -- particularly a spin glass -- through a sequence of meta-stable states. A distinctive feature of the systems considered by Bouchaud is that they exhibit the phenomena of \emph{ageing}, meaning that we can tell how long the system has been running by observing its present state. For one-dimensional versions of the model, which were first studied in \cite{FIN99} (in \cite{Bou92} the BTM was studied on the complete graph), the limited number of accessible sites means the property of ageing is intrinsically related to \emph{localisation}, namely that at certain times we can predict with high probability the state of the system (for further background, see \cite{BB03, FIN02}). It is the goal of this article to study this latter property in the special case of the BTM on the integers with regularly varying traps.

We now introduce the model of study, following the notation of \cite{CM15}. First, define a trapping landscape $\sigma=(\sigma_x)_{x \in \mathbb{Z}}$, which is a collection of independent and identically distributed (i.i.d.)\ strictly-positive random variables, built on a probability space with probability measure $\mathbf{P}$. Conditional on $\sigma$, the dynamics of the BTM are given by a continuous-time $\mathbb{Z}$-valued Markov chain $X=(X_t)_{t\geq 0}$, started from the origin, with transition rates
\begin{equation}\label{trates}
 w_{x \to y} = \begin{cases}
 \frac{1}{2 \sigma_x}, & \text{if }  y \sim x, \\
 0, & \text{otherwise,}
\end{cases}
\end{equation}
where $y \sim x$ means that $x$ and $y$ are nearest neighbours in $\mathbb{Z}$. We denote the law of $X$ conditional on $\sigma$, the so-called `quenched' law of the BTM, by $P_\sigma$. Throughout the article we will suppose that the trap distribution $\sigma_0$ satisfies
\begin{equation}\label{regvar}
\mathbf{P}\left(\sigma_0\geq u\right)= u^{-\alpha}, \qquad\forall u\geq 1,
\end{equation}
for some constant $\alpha\in(0,1]$. Whilst this is a strict assumption, we believe that, after making suitable adaptations to the arguments, one could deduce the same results under certain weaker conditions. For instance, in the case $\alpha \in (0, 1)$, it should be sufficient that the tail of the distribution of $\sigma_0$ be regularly varying with the same index. We exclude the parameter range $\alpha > 1$ since in this regime the BTM does not exhibit localisation, as explained below.

Specifically, the aim of this article is to establish the following quenched localisation behaviour of the BTM on the integers with regularly varying traps. On the one hand, we prove that for almost every trapping landscape there are arbitrarily large times at which the system is highly localised on one site, a site that can be described explicitly in terms of the trapping landscape. Conversely, we show that for almost every trapping landscape there are also arbitrarily large times at which the BTM is completely delocalised, i.e.\ no single site carries a prescribed amount of probability.

\begin{thm}\label{mainthm} For the BTM on the integers with a trapping distribution satisfying \eqref{regvar} for some $\alpha\in(0,1]$, it $\mathbf{P}$-a.s.\ holds that
\begin{align*}
\liminf_{t\rightarrow \infty}\sup_{x\in\mathbb{Z}}P_\sigma\left(X_t=x\right)=0,\qquad \limsup_{t\rightarrow \infty}\sup_{x\in\mathbb{Z}}P_\sigma\left(X_t=x\right)=1.
\end{align*}
\end{thm}

We now compare this result with previous studies of the quenched behaviour of the BTM that have appeared in the literature. In \cite{FIN99} (see also \cite{FIN02}) it was established that, for $\alpha \in (0, 1)$, the probability mass function of the BTM exhibits quenched localisation, in the sense that
\[\sup_{x \in \mathbb{Z}} P_{\sigma}\left(X_t=x\right) \not \rightarrow 0 \]
as $t \to \infty$. Theorem \ref{mainthm} strengthens the above localisation result by demonstrating that the supremum (indeed the $\ell_p$-norm, for any $p > 1$) of the probability mass function of the BTM fluctuates infinitely often between the bounds of $0$ and $1$.

One interesting consequence of Theorem \ref{mainthm} is to demonstrate a relatively sharp transition in the quenched behaviour of the BTM between the homogenised regime and a regime of strong disorder, as the tail of the trap distribution gets heavier. Recall that if $\sigma_0$ has finite mean (and in particular if \eqref{regvar} holds for some $\alpha > 1$), the BTM homogenises $\mathbf{P}$-a.s. In other words the BTM, rescaled diffusively, converges in distribution to Brownian motion $\mathbf{P}$-a.s. In Section \ref{sec:bm} we check that this, in turn, implies that, $\mathbf{P}$-a.s., as $t \to \infty$,
\begin{equation}\label{bmdeloc}
\sup_{x\in\mathbb{Z}}P_\sigma\left(X_t=x\right) \to 0.
\end{equation}
By contrast, the $\limsup$ part of Theorem \ref{mainthm} demonstrates that, if $\alpha \le 1$, there are arbitrarily large times, $\mathbf{P}$-a.s., at which the probability mass function of the BTM is in a maximally disordered state. Whether this remains true for \textit{any} $\sigma_0$ with infinite mean (perhaps under suitable regularity conditions) is an interesting question which we leave open.

We also note that the recent work \cite{CM15} investigated the result corresponding to Theorem \ref{mainthm} in the case of $\sigma_0$ with a slowly varying tail at infinity (roughly this is the $\alpha =0$ case; see \cite[Theorem 1.9]{CMspa}). In particular, it was demonstrated that in the one-sided case (i.e.\ restricting the BTM to the positive integers), there exist distributions of $\sigma_0$ such that
\[\liminf_{t\rightarrow \infty}\sup_{x\in\mathbb{Z}^+}P_\sigma\left(X_t=x\right)=\frac1N,\qquad\limsup_{t\rightarrow \infty}\sup_{x\in\mathbb{Z}^+}P_\sigma\left(X_t=x\right)=1 \]
for each $N \in \{2, 3, \ldots\}$. We suspect that the equivalent result also holds true in the two-sided case, for $N$ restricted to $\{3, 4, \ldots\}$. Whether there exist trap distributions $\sigma_0$ for which
\[\liminf_{t\rightarrow \infty}\sup_{x\in\mathbb{Z}}P_\sigma\left(X_t=x\right) = p \]
for arbitrary $p \in [0, 1/3]$ (or $p \in [0, 1/2]$ in the one-sided case) is also an interesting open question.

Our approach to establishing the $\limsup$ part of Theorem \ref{mainthm}, which we do in Section \ref{sec:comloc}, largely follows the argument in \cite{CM15}. On the other hand, our argument for the $\liminf$ part in Section \ref{sec:comdeloc} makes use of heat-kernel estimates, which is quite different from the approach taken for the equivalent bounds in \cite{CM15}.

Finally, we remark that as a by-product of our argument we establish bounds on the almost sure fluctuations in the max/sum ratio of i.i.d.\ sequences of random variables with common distribution $\sigma_0$. Again, one might compare to the integrable case, in which the $\limsup$ will tend to zero, and the slowly varying case, in which the $\liminf$ can be positive.

\begin{thm} \label{summax}
Assume $\sigma_0$ satisfies \eqref{regvar} for some $\alpha\in(0,1]$. Denote by $M_n$ and $S_n$ the maximum and sum respectively of the partial sequence $(\sigma_i)_{1 \le i \le n}$
\[    M_n := \max_{1 \le i \le n} \sigma_i , \qquad S_n := \sum_{1 \le i \le n} \sigma_i  .\]
Then it $\mathbf{P}$-a.s.\ holds that
\[\liminf_{n\rightarrow \infty}\frac{M_n}{S_n}=0, \qquad \limsup_{n\rightarrow \infty}\frac{M_n}{S_n}= 1.\]
\end{thm}

\section{Localisation on a single point}
\label{sec:comloc}

The aim of this section is to prove that at arbitrarily large times the BTM is highly localised, that is, to prove the $\limsup$ part of Theorem \ref{mainthm}. Our approach is to show that certain favourable configurations of the trapping landscape occur infinitely often $\mathbf{P}$-a.s., and moreover, when such favourable configurations arise, the BTM is highly localised on a single site at a certain time.

We first introduce an $\varepsilon\in(0,1)$ that will act as our error threshold. Unless explicitly stated, $\varepsilon$ will remain fixed throughout this section, and as such we shall suppress the explicit dependence of other notation on $\varepsilon$. To define the favourable configurations, we introduce the scales, for $n \in \mathbb{N}$,
\begin{equation}\label{andef}
a_n:=\lfloor e^{2 n \log n}\rfloor,\qquad b_n:=  \lceil \varepsilon^{-1} a_n \rceil.
\end{equation}
Note that we have chosen $a_n$ specifically so that the ratio $a_{n-1} / a_n \sim  n^{-2}  \to 0$ is a summable sequence. Further, recalling the notation for the one-sided maximum and sum processes $M_n$ and $S_n$ from the statement of Theorem \ref{summax}, we introduce the two-sided sum process
\[   \bar S_n := \sum_{-n \le i \le n} \sigma_i , \]
and define the events, for $n \in \mathbb{N}$,
\[\mathcal{E}_n:=\left\{M_{a_n}>\varepsilon^{-2}a_n^{1/\alpha} \ell_\alpha(a_n),\: \bar S_{b_n}  - M_{a_n} < 3 \varepsilon^{-1} a_n^{1/\alpha} \ell_\alpha(a_n)  \right\},\]
where
\[ \ell_\alpha(n) := \begin{cases}
1, & \text{if } \alpha\in(0,1), \\
\log n, & \text{if } \alpha=1,
\end{cases} \]
is a logarithmic correction in the case $\alpha = 1$.

We show that, for each $\varepsilon\in(0,1)$, the events $\mathcal{E}_n$ occur infinitely often $\mathbf{P}$-a.s. (see Lemma \ref{enlem}). Moreover, we show that, when the event $\mathcal{E}_n$ occurs, at the time
\[t_n:= \varepsilon^{-2} a_n^{1+1/\alpha} \ell_\alpha(a_n) \]
the BTM is completely localised, up to an $\varepsilon$-dependent error, on the site achieving the maximum $M_{a_n}$ (see Corollary \ref{crllocal}), a site that we shall denote by
\[x_n:=\mathrm{arg\:max}_{ 1 \le i \le a_n} \sigma_{i}.\]

To establish that the events $\mathcal{E}_n$ occur infinitely often, we start by proving the corresponding result for a closely related sequence of independent events. In particular, define the collections of intervals $(I_n)_{n\geq 1}$ and $(J_n)_{n\geq 1}$ by setting $I_1 := (0,a_1]$, $J_1 := [-b_1, 0] \cup (a_1,b_1] $, and, for $n\geq 2$,
\[I_n:=(b_{n-1},a_n],\qquad J_n:=[-b_n, -b_{n-1}) \cup (a_n,b_n] . \]
Note that $\cup_{n=1}^\infty (I_n\cup J_n)=\mathbb{Z}$, and also $I_n,J_n$, $n\geq 1$, are all disjoint. For a subset $I\subseteq \mathbb{Z}$, we write
\[S(I):=\sum_{i\in I}\sigma_i,\qquad M(I):=\max_{i\in I}\sigma_i.\]
Then define the events
\[\mathcal{A}_n:=\left\{M(I_n)>\varepsilon^{-2} a_n^{1/\alpha} \ell_\alpha(a_n),\: S(I_n\cup J_n)-M(I_n)< 2 \varepsilon^{-1} a_n^{1/\alpha} \ell_\alpha(a_n)\right\}.\]
Importantly, we observe that the disjointness of the intervals involved in their definition readily yields that these events are independent. We use this fact in the proof of the following result.

\begin{lmm} \label{anlem} As $n \to \infty$, it $\mathbf{P}$-a.s.\ holds that $\mathcal{A}_n$ occurs infinitely often.
\end{lmm}
\begin{proof}  By the independence of $(\mathcal{A}_n)_{n\geq 1}$ and the second Borel-Cantelli lemma, it will suffice to show that
\begin{equation}\label{sum}
\sum_{n=1}^\infty\mathbf{P}\left(\mathcal{A}_n \right)=\infty.
\end{equation}
Since we have a continuous distribution for $\sigma_0$, it holds that
\begin{align*}
& \mathbf{P}  \left(\mathcal{A}_n \right) =\sum_{i\in I_n}\mathbf{P}\left(\mathcal{A}_n, \:M(I_n)=\sigma_i\right)\\
&=(a_n-b_{n-1})\mathbf{P}\left(\sigma_{a_n}>\varepsilon^{-2} a_n^{1/\alpha} \ell_\alpha(a_n),\:S(I_n\cup J_n  \backslash\{a_n\})<2 \varepsilon^{-1} a_n^{1/\alpha} \ell_\alpha(a_n)\right)\\
& \ge (a_n-b_{n-1})\mathbf{P}\left(\sigma_{0}>\varepsilon^{-2}a_n^{1/\alpha} \ell_\alpha(a_n) \right)\mathbf{P}\left(S_{2 b_n}<2  \varepsilon^{-1} a_n^{1/\alpha} \ell_\alpha(a_n)\right).
\end{align*}
From \eqref{regvar} it is easy to check that the first of these probabilities satisfies
\[ (a_n-b_{n-1})\mathbf{P}\left(\sigma_{0}>\varepsilon^{-2}a_n^{1/\alpha}\ell_\alpha(a_n)\right) \sim \varepsilon^{2\alpha}  \ell_\alpha(a_n)^{-\alpha} \]
as $n \rightarrow \infty$. To control the second probability, we treat the cases $\alpha \in (0, 1)$ and $\alpha = 1$ separately. In the case $\alpha \in (0, 1)$, it is well-known that (see \cite[p. 62, Table 2.1]{UZ99}, for example), as $n \to \infty$, $n^{-1/\alpha}S_n \Rightarrow \Xi_\alpha$ in distribution, where $\Xi_\alpha$ is a random variable with a maximally-asymmetric $\alpha$-stable law supported on $(0, \infty)$. Hence
\[\mathbf{P}\left(S_{2b_n}< 2 \varepsilon^{-1} a_n^{1/\alpha}  \right)  \rightarrow \mathbf{P}\left(\Xi_\alpha< (2\varepsilon^{-1})^{1-1/\alpha}\right) \]
as $n\rightarrow\infty$. Thus we find that
\[\mathbf{P}\left(\mathcal{A}_n\right)\rightarrow \varepsilon^{2\alpha}\mathbf{P} \left(\Xi_\alpha< (2\varepsilon^{-1})^{1-1/\alpha}\right)  >0,\]
and the result at \eqref{sum} follows. In the case $\alpha = 1$, it is instead the case that (again, see \cite[p. 62, Table 2.1]{UZ99}, for example), as $n \to \infty$, $ n^{-1} (S_n - n \log n ) \Rightarrow \Xi_1$ in distribution, where $\Xi_1$ is a random variable with a maximally-symmetric $1$-stable law supported on~$\mathbb{R}$. Hence
\[\mathbf{P}\left(S_{2b_n}< 2 \varepsilon^{-1}  a_n \log a_n \right)  \rightarrow \mathbf{P}\left(\Xi_1< - \log (2/ \varepsilon)  \right)  > 0\]
as $n\rightarrow\infty$. Thus we find that, as $n \to \infty$,
\[\mathbf{P}\left(\mathcal{A}_n\right) \sim  c \log(a_n)^{-1} \sim c / (2 n \log n) ,\]
for some constant $c > 0$, and so the result at \eqref{sum} also follows in this case.
\end{proof}

\begin{lmm}\label{enlem}
As $n \to \infty$, it $\mathbf{P}$-a.s.\ holds that $\mathcal{E}_n$ occurs infinitely often.
\end{lmm}
\begin{proof}
We start by defining the event
\[\mathcal{B}_n :=\left\{\bar S_{b_{n-1}}<  a_n^{1/\alpha} \ell_\alpha(a_n) \right\} ,\]
which we claim holds eventually, $\mathbf{P}$-a.s.  To prove this, note that it is an elementary computation to deduce from \eqref{regvar} that
\[1 - \mathbf{E}\left(e^{-\theta\sigma_0}\right)   \sim  c_\alpha \theta^\alpha \ell_\alpha(\theta^{-1}) \]
as $\theta\rightarrow0$, where
\[  c_\alpha :=  \begin{cases}
 \Gamma(1-\alpha) , & \text{if } \alpha \in (0, 1), \\
1 , & \text{if } \alpha = 1,
\end{cases}\]
with $\Gamma$ the usual gamma function. Thus, for any $c_n$ such that
\[   \lim_{n \to \infty}  n c_n^{-\alpha} \ell_\alpha(c_n) = 0 , \]
applying Markov's inequality we have,
\begin{eqnarray*}
\mathbf{P}\left( S_n >  c_n \right) & \leq&  \frac{1-\mathbf{E}\left(e^{-c_n^{-1}
 S_n }\right)}{1-e^{-1}}\\
&=&(1-e^{-1})^{-1}\left(1-\mathbf{E}\left(e^{-c_n^{-1} \sigma_0}\right)^{n}\right)\\
&\sim& c_\alpha (1-e^{-1})^{-1}  n c_n^{-\alpha} \ell_\alpha(c_n)  .
\end{eqnarray*}
Finally note that there is a constant $c > 0$ such that, as $n \to \infty$, eventually
\[  \frac{(2b_{n-1} + 1)  \ell_\alpha(a_n^{1/\alpha} \ell_\alpha(a_n))  }{a_n \ell_\alpha(a_n)^\alpha}  < c n^{-2} \to 0 , \]
where we have used the fact that $a_{n-1}/a_n \sim n^{-2}$. Hence we deduce that there exists a constant $c > 0$ such that, as $n \to \infty$, eventually
\[ \mathbf{P}\left( \bar S_{b_{n-1}}  > a_n^{1/\alpha} \ell_\alpha(a_n) \right) = \mathbf{P} \left( S_{2b_{n-1} + 1}  > a_n^{1/\alpha} \ell_\alpha(a_n)   \right) < c n^{-2}  , \]
and by the Borel-Cantelli lemma the claim is proved.

Now, it is a consequence of Lemma \ref{anlem} and the conclusion of the previous paragraph that $\mathcal{A}_n\cap\mathcal{B}_n$ occurs infinitely often, $\mathbf{P}$-a.s. Thus to complete the proof it will suffice to show that $\mathcal{A}_n \cap\mathcal{B}_n \subseteq \mathcal{E}_n$. However, this is straightforward, since on $\mathcal{A}_n\cap\mathcal{B}_n$, we have that
\[M_{a_n}=\max\{M_{b_{n-1}},M(I_n)\} \geq M(I_n) >\varepsilon^{-2}a_n^{1/\alpha} \ell_\alpha(a_n),\]
and also
\begin{align*}
\bar S_{b_n}-M_{a_n} \leq \bar S_{b_{n-1}}+S(I_n\cup J_n )-M(I_n) < (2 \varepsilon^{-1} + 1) a_n^{1/\alpha} \ell_\alpha(a_n)
\end{align*}
as required, since $\varepsilon < 1$.
\end{proof}

We now proceed to establish that, on the event $\mathcal{E}_n$, at the time $t_n$ the BTM is completely localised, up to an $\varepsilon$-dependent error, on the site $x_n$. We first state a general localisation result that is valid for arbitrary times $t > 0$ and deterministic trapping landscapes $\sigma$, before specialising to the time $t_n$ and invoking the properties of the trapping landscape contained in $\mathcal{E}_n$.

\begin{lmm}\label{locallem}
Let $\sigma$ be a deterministic strictly-positive trapping landscape, and denote by $X$ the continuous-time $\mathbb{Z}$-valued Markov chain, started from $0$, with transition rates given by \eqref{trates}, with $P_\sigma$ its law.
Then, for each pair of sites $0 < x < y$ and each time $t > 0$,
\[P_\sigma\left( X_t = x \right) \ge  \left(\frac{y}{x+y} -  t^{-1} x  \sum_{-y < z < x} \sigma_z \right)_+ \left( \frac{\sigma_x}{\sum_{-y  \leq z \leq y } \sigma_z}   - \frac{t}{(y-x)\sigma_x} \right)_+ , \]
where $\alpha_+$ denotes the positive part of $\alpha\in\mathbb{R}$.
\end{lmm}

\begin{proof}
We prove Lemma \ref{locallem} in a similar manner to the equivalent results in \cite{CM15}, which were established in the one-sided case. Here we adapt these arguments to the two-sided case.

For a site $z \in \mathbb{Z}$, denote by $\tau_z$ the first hitting time by $X$ of $z$, i.e.\ $\tau_z := \inf \{s \ge 0 : X_s = z \}$, and further denote by $P_\sigma^z$ the law of the Markov chain $X$ started from $z$. By applying the strong Markov property at the (almost-surely finite) stopping time $\tau_x$, we have that
\begin{equation}
P_\sigma\left( X_t = x \right)=\int_0^tP^x_\sigma\left( X_{t-s} = x\right)P_\sigma\left( \tau_x\in ds\right)\geq  P_\sigma\left( \tau_x\leq t\right)\inf_{s\leq t}P^x_\sigma\left( X_{s} = x\right).\label{tttt}
\end{equation}
We will bound these probabilities using the methods of \cite{CM15}.

In particular, to bound the first of the terms in the lower bound at \eqref{tttt}, note that
\[ P_\sigma\left( \tau_x\leq t\right)\geq P_\sigma\left( \tau_x\leq t \wedge  \tau_{-y} \right)\geq P_\sigma\left( \tau_x < \tau_{-y}\right)- P_\sigma\left(  \tau_x \wedge \tau_{-y} \geq t \right).\]
By basic properties of random walks, it holds that $P_\sigma( \tau_x < \tau_{-y}) = y/(x+y)$. Moreover, by \cite[Proposition 2.1]{CM15}, we have the following upper bound
\[  P_\sigma( \tau_x \wedge \tau_{-y}  \ge t) \le  t^{-1} x \sum_{-y < z < x} \sigma_z , \]
and so we conclude
\[ P_\sigma\left( \tau_x\leq t\right)\geq \left(\frac{y}{x+y} -  t^{-1} x  \sum_{-y < z < x} \sigma_z\right)_+ .\]

To bound the second of the terms in the lower bound at \eqref{tttt}, we observe
\begin{eqnarray}
P^x_\sigma\left( X_{s} = x\right)&\geq& P^x_\sigma\left( X_s = x, \tau_{-y} \wedge \tau_y >  t \right)\nonumber\\
&\geq& P^x_\sigma \left( \tau_{-y} \wedge \tau_y  >  t  \right) - P^x_\sigma\left( X_s \neq x,  \tau_{-y} \wedge \tau_y  >  t \right).\label{t5}
\end{eqnarray}
For the first term here we can apply the following lower bound (see \cite[Proposition 2.2]{CM15})
\[  P^x_\sigma(  \tau_{-y} \wedge \tau_y    > t) \ge  1 - \frac{t}{(y-x) \sigma_x } .\]
Towards bounding the remaining probability in \eqref{t5}, we define the continuous-time $[-y, y]$-valued Markov chain $\hat X$, started from $x$, with transition rates given by \eqref{trates} (interpreting $\sim$ as denoting nearest neighbours on $[-y, y] \cap\mathbb{Z}$), and let $\hat P^x_\sigma$ be its law. Then it is clear that $(\hat X_{s \wedge \tau_y \wedge \tau_{-y}} )_{s \ge 0}$ has the same distribution as the chain $(X_{s \wedge \tau_y \wedge \tau_{-y}} )_{s \ge 0}$ started from $x$. Hence, for any $s\leq t$, we have that
\begin{align*}
P^x_\sigma \left( X_s \neq x ,   \tau_y \wedge \tau_{-y} > t \right)    & =   \hat P^x_\sigma \left( \hat X_s \neq x,  \tau_y \wedge \tau_{-y} > t \right)  \\
&  \le  \hat P^x_\sigma \left( \hat X_s \neq x  \right) \le  1 - \inf_{u \ge 0}  \hat P^x_\sigma \left( \hat{X}_u= x \right) .
\end{align*}
Finally note that, since $\hat P^x_\sigma(\hat X_s = x )$ is decreasing in $s$, the process $\hat X$ satisfies
\[ \inf_{s \ge 0} \hat P^x_\sigma(\hat X_s = x ) = \lim_{s\to \infty} \hat P^x_\sigma(\hat X_s = x )= \frac{\sigma_x}{\sum_{-y  \leq z \leq y } \sigma_z} , \]
where the second inequality follows from the ergodicity of the Markov chain in question, and a computation of its invariant distribution from the detailed balance equations. Combining the above results establishes that
\[\inf_{s\leq t}P^x_\sigma\left( X_{s} = x\right)\geq
\left(\frac{\sigma_x}{\sum_{-y  \leq z \leq y } \sigma_z}   - \frac{t}{(y-x)\sigma_x}\right)_+,\]
which completes the proof.
\end{proof}

Applying the previous result with the setting $x := x_n, y := b_n$ and $t := t_n$, together with the definition of $\mathcal{E}_n$, we readily deduce the following.

\begin{crl}\label{crllocal}
On the event $\mathcal{E}_n$, it holds that
\begin{align} \label{loc}
P_\sigma\left(X_{t_n}=x_n\right) >  \left(\frac{1}{1+\varepsilon}-3\varepsilon\right)_+\left( \frac{1}{1+3\varepsilon} - \frac{\varepsilon}{1-\varepsilon}\right)_+.
\end{align}
The right-hand side of \eqref{loc} can be written $1 - O(\varepsilon)$ as $\varepsilon \to 0$.
\end{crl}

Observe that, for any $\varepsilon \in (0, 1)$, the times $t_n \to \infty$ as $n \to \infty$. Hence, since $\varepsilon \in (0, 1)$ was arbitrary, as a simple consequence of Lemma \ref{enlem} and Corollary \ref{crllocal} we obtain the $\limsup$ part of Theorem \ref{mainthm}. Furthermore, note that on $\mathcal{E}_n$ we have that
\[1 \ge \frac{M_{a_n}}{S_{a_n}} \ge \frac{M_{a_n}}{\bar S_{b_n}} = \frac{1}{1+(\bar S_{b_n} - M_{a_n})/M_{a_n}}  > \frac{1}{1 + 3\varepsilon} .\]
Hence, since $a_n \rightarrow \infty$ and $\varepsilon \in (0, 1)$ was arbitrary, we also deduce from Lemma \ref{enlem} the $\limsup$ part of Theorem \ref{summax}.

\section{Complete delocalisation}
\label{sec:comdeloc}

In this section we prove that the BTM is completely delocalised at arbitrarily large times, that is, we prove the $\liminf$ part of Theorem \ref{mainthm}. As in Section \ref{sec:comloc}, our approach is to show that certain favourable configurations of the trapping landscape occur infinitely often $\mathbf{P}$-a.s., and moreover, when such favourable configurations arise, the BTM is highly delocalised.

Throughout this section we again introduce an $\varepsilon\in(0,1)$ to act as our error threshold. We also introduce a parameter $K \in \mathbb{N}$ which measures the extent of the spread of the probability mass function of the BTM. Again, unless explicitly stated, both $\varepsilon$ and $K$ will remain fixed throughout this section, and as such we suppress the explicit dependence of other notation on $\varepsilon$ and $K$.

To define the favourable configurations, recall the scale $a_n$ from \eqref{andef} and further define, for $k \in \mathbb{Z}$, the evenly spaced sites $ a_{n, k} := k a_n$,
the corresponding intervals
\[I_{n,k}:=[a_{n,k},a_{n,k+1}),\]
and the events, for $n \in \mathbb{N}$,
\[\mathcal{E}_n:= \bigcap_{k \in [-K, K]} \left\{S(I_{n,k})\in\left(\frac{1}{2}a_n^{1/\alpha}\ell_\alpha(a_n),2a_n^{1/\alpha}\ell_\alpha(a_n)\right),\:M(I_{n,k})<\varepsilon a_n^{1/\alpha} \ell_\alpha(a_n) \right\} .\]
We will show that, for each $\varepsilon \in (0, 1)$ and $K \in \mathbb{N}$, the events $\mathcal{E}_n$ occur infinitely often $\mathbf{P}$-a.s.\ (see Lemma \ref{enlem2}). Moreover, we show that, when the event $\mathcal{E}_n$ occurs, at the time
\[t_n:=12a_n^{1+1/\alpha}\ell_\alpha(a_n)\]
no site carries a non-negligible ($(\varepsilon,K)$-dependent) proportion of the probability mass of the BTM (see Lemma \ref{farbound}).

To establish that the events $\mathcal{E}_n$ occur infinitely often, we again start by proving the corresponding result for a closely related sequence of independent events. In particular, define the collection of intervals $(\tilde{I}_{n, k})_{n \ge 2, k \in [-K,K]}$ by setting,
\[ \tilde{I}_{n,0} :=[(K+1)a_{n-1},a_n), \qquad \tilde{I}_{n, -1} := [a_n, -Ka_{n-1}) \]
and $\tilde{I}_{n,k}={I}_{n,k}$ for $k\in [-K, K] \setminus \{-1, 0\}$. We then set
\[\mathcal{A}_n:=  \bigcap_{k \in [-K, K]} \left\{S(\tilde{I}_{n,k})\in\left(\frac{1}{2}a_n^{1/\alpha}\ell_\alpha(a_n),\frac{3}{2}a_n^{1/\alpha}\ell_\alpha(a_n)\right),\:M(\tilde{I}_{n,k})< \varepsilon a_n^{1/\alpha} \ell_\alpha(a_n) \right\} .\]
Observe that the intervals $(\tilde{I}_{n, k})_{n \ge 2, k \in [-K, K]}$  are distinct for sufficiently large $n$. Hence, as $n \to \infty$ the events $\mathcal{A}_n$ are eventually independent. We use this fact in the proof of the following result.

\begin{lmm} \label{anlem2}
As $n \to \infty$, it $\mathbf{P}$-a.s.\ holds that $\mathcal{A}_n$ occurs infinitely often.
\end{lmm}
\begin{proof}
By the eventual independence of $(\mathcal{A}_n)_{n\geq 2}$ and the second Borel-Cantelli lemma, it suffices to show that $\sum_{n=2}^\infty\mathbf{P}(\mathcal{A}_n)=\infty$. It is well-known (\cite{Skoro}) that as $n \to \infty$
\[  \left( n^{-1/\alpha} \ell_\alpha(n)^{-1} S_{\lfloor nt \rfloor} \right)_{t \ge 0} \Rightarrow   \left(\Xi_\alpha(t) \right)_{t \ge 0}  \]
weakly in the space $D(\mathbb{R}^+)$ of real-valued c\`{a}dl\`{a}g functions equipped with the standard Skorohod $J_1$ topology (see \cite{Whitt02} for the definition), where $\Xi_\alpha(t)$ denotes an $\alpha$-stable subordinator for $\alpha \in (0, 1)$, and $\Xi_1(t) := t$. Notice that the functionals $f\mapsto f(1)$ and $f \mapsto \sup_{t \in [0, 1]} \Delta f(t)$, where $\Delta f (t):=f(t)-f(t^-)$ denotes the jump in the function $f$ at the time $t$, are both continuous in the $J_1$ topology for functions that are continuous at $t=1$, and moreover that all fixed times are continuity times for $\Xi_\alpha$. Hence it follows that
\[\mathbf{P}\left(\mathcal{A}_n \right) \rightarrow \mathbf{P}\left(\Xi_\alpha(1)\in(1/2,3/2),\:\sup_{t\in[0,1]}\Delta\Xi_\alpha(t)<\varepsilon\right)^{2K+1} > 0\]
as desired.
\end{proof}

\begin{lmm}\label{enlem2}
As $n \to \infty$, it $\mathbf{P}$-a.s.\ holds that $\mathcal{E}_n$ occurs infinitely often.
\end{lmm}
\begin{proof}
We start by defining the event, for $n \ge 2$,
\[\mathcal{B}_n :=\left \{ S( -Ka_{n-1}, (K+1)a_{n-1} ) <  \left(\frac{1}{2}\wedge\varepsilon\right) a_n^{1/\alpha} \ell_\alpha(a_n) \right\} ,\]
which we claim holds eventually, $\mathbf{P}$-a.s.  To prove this, recall from the proof of Lemma~\ref{enlem} that, for any $c_n$ such that $ \lim_{n \to \infty}  n c_n^{-\alpha} \ell_\alpha(c_n) = 0$, there exists a constant $c > 0$ such that, as $n \to \infty$, eventually $\mathbf{P}( S_n >  c_n) < c  n c_n^{-\alpha} \ell_\alpha(c_n^{-1})$. Now consider that there is a constant $c > 0$ such that, as $n \to \infty$, eventually
\[  \frac{a_{n-1}  \ell_\alpha(\left(\frac{1}{2}\wedge\varepsilon\right) a_n^{1/\alpha} \ell_\alpha(a_n))  }{ a_n \ell_\alpha(a_n)^\alpha }  <  c  n^{-2} \to 0 , \]
where we use the fact that $a_{n-1}/a_n \sim n^{-2}$. Hence we deduce that there exists a constant $c > 0$ such that, as $n \to \infty$, eventually
\begin{align*}
 &\mathbf{P}\left( S( -Ka_{n-1}, (K+1)a_{n-1} )>  \left(\frac{1}{2}\wedge\varepsilon\right) a_n^{1/\alpha} \ell_\alpha(a_n)  \right) \\
 & \qquad = \mathbf{P} \left( S_{2(K+1) a_{n-1} }  > \left(\frac{1}{2}\wedge\varepsilon\right) a_n^{1/\alpha} \ell_\alpha(a_n)   \right)  < c n^{-2}  ,
 \end{align*}
and by the Borel-Cantelli lemma the claim is proved.

Now, it is a consequence of Lemma \ref{anlem2} and the previous paragraph that $\mathcal{A}_n \cap\mathcal{B}_n$ occurs infinitely often, $\mathbf{P}$-a.s. Thus to complete the proof it will suffice to show that $\mathcal{A}_n \cap\mathcal{B}_n \subseteq \mathcal{E}_n$. However, this is straightforward for, since on $\mathcal{A}_n\cap\mathcal{B}_n$ each $k = -1, 0$ satisfies
\[ M( I_{n, k})  \leq  \max\left\{  M(\tilde{I}_{n, k}),  S( -Ka_{n-1}, (K+1)a_{n-1} )  \right\} < \varepsilon a_n^{1/\alpha} \ell_\alpha(a_n),\]
and
\begin{align*}
   S({I}_{n, k}) \in \left( S(\tilde{I}_{n, k})  ,  S(\tilde{I}_{n, k}) +  S( -Ka_{n-1}, (K+1)a_{n-1} ) \right) \in \left( 1/2 , 2\right) a_n^{1/\alpha} \ell_\alpha(a_n),
\end{align*}
and moreover the conditions on $M(I_{n, k})$ and $S(I_{n, k})$ for each $k\in [-K, K] \setminus \{-1 , 0\}$ are identical in the events $\mathcal{A}_n$ and $\mathcal{E}_n$.
\end{proof}

We now proceed to study the behaviour of the BTM on the event $\mathcal{E}_n$. In particular, our first aim is to show that at time $t_n$ no site in the interval $[a_{n,-K+1},a_{n,K}]$ carries significant mass.

\begin{lmm}\label{nearbound} If $\mathcal{E}_n$ holds, then
\[\sup_{x\in[a_{n,-K+1},a_{n,K}]}P_\sigma\left(X_{t_n}=x\right) < 4\varepsilon.\]
\end{lmm}
\begin{proof}  We first introduce the notation
\[V(x,r):=S([x-r+1,x+r-1])\]
and denote the quenched heat kernel of the Markov chain $X$ by
\begin{equation}\label{hkdef}
p^\sigma_t(x,y)=\sigma_y^{-1}P_\sigma^x\left(X_t=y\right),
\end{equation}
where again we write $P_\sigma^x$ to denote the law of $X$ started from $x$. Throughout the proof we suppose that $\mathcal{E}_n$ holds, and we note that on this event
\begin{equation}\label{measubound}
\frac{1}{2}a_n^{1/\alpha}\ell_\alpha(a_n) <  V(x,a_n) < 6a_n^{1/\alpha}\ell_\alpha(a_n),\qquad \forall x\in [a_{n,-K+1},a_{n,K}].
\end{equation}
Next, in this one-dimensional setting, it is possible to check by applying the argument of \cite[Proposition 4.1]{Kum} that $p^\sigma_{2a_nV(x,a_n)}(x,x)\leq 2/V(x,a_n)$. Hence, since $(p^\sigma_t(x,x))_{t\geq 0}$ is decreasing in $t$,
it follows from (\ref{measubound}) that
\begin{equation}\label{ondiag}
p^\sigma_{t_n}(x,x) < \frac{4}{a_n^{1/\alpha}\ell_\alpha(a_n)},\qquad \forall x\in [a_{n,-K+1},a_{n,K}].
\end{equation}
In conjunction with the Cauchy-Schwarz bound
\[p^\sigma_t(0,x)\leq \sqrt{p^\sigma_t(0,0)p^\sigma_t(x,x)},\qquad \forall x\in\mathbb{Z},\:t>0,\]
and the estimate $\max_{x\in [a_{n,-K+1},a_{n,K}]}\sigma_x < \varepsilon a_n^{1/\alpha}\ell_\alpha(a_n)$ that holds on $\mathcal{E}_n$, \eqref{ondiag} implies the result.
\end{proof}

We now extend the bound of the previous lemma to hold uniformly over the entire integer lattice, by checking that on $\mathcal{E}_n$ the Markov chain $X$ does not exit the interval $[a_{n,-K+1},a_{n,K}]$ quickly.

\begin{lmm}\label{farbound} If $\mathcal{E}_n$ holds, then
\[\sup_{x\in\mathbb{Z}}P_\sigma\left(X_{t_n}=x\right) < 4\varepsilon + b_K,\]
where $(b_k)_{k\geq 1}$ is a deterministic sequence of positive numbers such that $b_k\rightarrow 0$ as $k\rightarrow\infty$.
\end{lmm}
\begin{proof}
First we introduce the two-sided hitting time
\[\tau(x,r):=\inf\{t\geq 0:\:|X_t-x|\geq r\} ,\]
and note that, in light of Lemma \ref{nearbound}, it suffices to show that, on $\mathcal{E}_n$,
\begin{align}
\label{eqK}
P_\sigma\left(\tau(0,(K-1)a_n)\leq t_n\right) \leq   b_K ,
\end{align}
where $(b_k)_{k\geq 1}$ is a deterministic sequence of positive numbers such that $b_k \to 0$.

To bound the probability at \eqref{eqK}, set $\tilde{t}_n:= t_n/24 = \frac{1}{2}a_n^{1+1/\alpha}\ell_\alpha(a_n)$, and write $\tau_0=0$ and \[\tau_{i+1}:=\inf\left\{t\geq \tau_i:\:X_t\in 2a_n\mathbb{Z}\backslash\{X_{\tau_i}\}\right\}.\]
Then we may bound \eqref{eqK} as follows:
\begin{eqnarray}
P_\sigma\left(\tau(0,(K-1)a_n)\leq t_n\right)&\leq &P_\sigma\left(\sum_{i=0}^{\lfloor (K-1)/2\rfloor-1}\tau_i\leq t_n\right)\nonumber\\
&\leq& P_\sigma\left(\sum_{i=0}^{\lfloor (K-1)/2\rfloor-1}\mathbf{1}_{\{\tau_i\geq \tilde{t}_n\}}\leq24\right) .\label{ub111}
\end{eqnarray}
Writing $P^x_\sigma$ for the law of $X$ started from $x$, it now suffices to show that on~$\mathcal{E}_n$, for every $x\in[a_{n,-K+2},a_{n,K-1}]\cap a_n\mathbb{Z}$,
\begin{align}
\label{eqK2}
P^x_\sigma\left(\tau(x,2a_n)\geq \tilde{t}_n \right) > 1/32 .
\end{align}
Indeed, \eqref{ub111} is then bounded above by $b_K$, as defined by setting
\[ b_k := \mathbb{P}\left(\mathrm{Bin}\left(\lfloor (k-1)/2\rfloor,1 /32 \right)\leq 24\right), \]
where $\mathrm{Bin}(N,p)$ a binomial random variable with parameters $N$ and $p$, which is a choice that clearly yields $(b_k)_{k\geq1}$ is a null sequence.

To establish \eqref{eqK2}, first let $E^x_\sigma$ denote the expectation corresponding to the law $P^x_\sigma$. By considering the number of visits to each vertex by the jump process of $X$, we have that: for every $x\in\mathbb{Z}$, $r\in\mathbb{N}$, $y\in(x-r,x+r)$,
\[E^y_\sigma(\tau(x,r))=\sum_{i=x-r+1}^{x+r-1}P^y_\sigma(\tau_i<\tau(x,r))\times P^i_\sigma(\tau_i^+>\tau(x,r))^{-1} \times \sigma_i,\]
where we again write $\tau_i$ to represent the first hitting time by $X$ of $i$, and write $\tau_i^+:= \inf \{s \ge \inf\{t:\:X_t\neq i\} : X_s = i \}$ to represent the first return time. Now, it is an elementary computation to check that
\[P^y_\sigma(\tau_i<\tau(x,r))=\min\left\{\frac{x+r-y}{x+r-i},\frac{y-x+r}{i-x+r}\right\},\]
\[P^i_\sigma(\tau_i^+>\tau(x,r))=\frac{1}{2}\left(\frac{1}{i-x+r}+\frac{1}{x+r-i}\right),\]
and so we obtain:
\[E^y_\sigma(\tau(x,r))=2\sum_{i=x-r+1}^{x+r-1}\min\left\{\frac{x+r-y}{x+r-i},\frac{y-x+r}{i-x+r}\right\}\times
\left(\frac{1}{i-x+r}+\frac{1}{x+r-i}\right)^{-1} \sigma_i.\]
By the definition of $\mathcal{E}_n$ we readily deduce from this bound that
\[E^y_\sigma\left(\tau(x,2a_n) \right) < 16a_n^{1+1/\alpha}\ell_\alpha(a_n)\]
and
\[E^x_\sigma\left(\tau(x,2a_n) \right) > a_n^{1+1/\alpha}\ell_\alpha(a_n)\]
for every $x\in[a_{n,-K+2},a_{n,K-1}]\cap a_n\mathbb{Z}$, $y\in (x-2a_n,x+2a_n)$. Thus, applying the Markov property at time $t$, we deduce
\begin{eqnarray*}
 a_n^{1+1/\alpha}\ell_\alpha(a_n)& < & E^x_\sigma\left(\tau(x,2a_n) \right)\\
 &\leq& t+ P^x_\sigma\left(\tau(x,2a_n)\geq t \right)16a_n^{1+1/\alpha}\ell_\alpha(a_n)
\end{eqnarray*}
for every $x\in[a_{n,-K+2},a_{n,K-1}]\cap a_n\mathbb{Z}$. Setting $t=\tilde{t}_n$ completes the proof. \end{proof}

Putting together Lemmas \ref{enlem2} and \ref{farbound}, recalling that  $\varepsilon \in (0, 1)$ and $K \in \mathbb{N}$ were arbitrary, and noting that $t_n\rightarrow\infty$, we have thus established the $\liminf$ part of Theorem~\ref{mainthm}. Furthermore, note that on $\mathcal{E}_n$ we have that
\[\frac{M_{a_n}}{S_{a_n}} < 2\varepsilon.\]
Hence we also deduce from Lemma \ref{enlem2} the $\liminf$ part of Theorem \ref{summax}.

\section{Delocalisation for traps with finite expectation}\label{sec:bm}

In this section, we deal with the case when $\mathbf{E}\sigma_0<\infty$, and in particular establish~(\ref{bmdeloc}). In this case, it is known that, for $\mathbf{P}$-a.e.\ realisation of the trapping environment, we have that
\begin{equation}\label{path}
(\varepsilon X_{t/\varepsilon^2})_{t\geq 0}\rightarrow (B_t)_{t\geq 0}
\end{equation}
in distribution in $D([0,\infty),\mathbb{R})$, where up to a deterministic linear time change $B=(B_t)_{t\geq 0}$ is a standard one-dimensional Brownian motion \cite[Theorem 2.10]{BCCR}. Moreover, by the strong law of large numbers we $\mathbf{P}$-a.s.\ have that $\varepsilon\sum_{x\in\mathbb{Z}}\sigma_{x}\delta_{\varepsilon x}\rightarrow \mathbf{E}(\sigma_0)\lambda$ vaguely, where $\delta_x$ is the probability measure placing all its mass at $x$, and $\lambda$ is Lebesgue measure on $\mathbb{R}$. It follows that we can apply the local limit theorem of \cite[Theorem 1]{CH} to deduce that, $\mathbf{P}$-a.s.,
\[\lim_{t\rightarrow\infty}\max_{x\in\mathbb{Z}:\:|x|\leq R\sqrt{t}}\sqrt{t}\left|p^\sigma_t(0,x)-q(x/\sqrt{t})\right|= 0,\]
where $p^\sigma$ is the discrete heat kernel defined as at \eqref{hkdef}, and $q$ is the density of $B_1$ with respect to $\mathbf{E}(\sigma_0)\lambda$. (For this, it is useful to note that in this setting the Euclidean metric coincides with the resistance metric, where we consider $\mathbb{Z}$ as an electrical network with unit conductances between nearest neighbours. This means we can immediately apply \cite[Proposition 14]{CH} to check the equicontinuity of the discrete heat kernels under the relevant scaling.) In particular, we obtain that, $\mathbf{P}$-a.s., for large $t$,
\[\max_{x\in\mathbb{Z}:\:|x|\leq R\sqrt{t}}P_\sigma\left(X_t=x\right)\leq 2\sup_{x\in\mathbb{R}:\:|x|\leq R}q(x)\max_{x\in\mathbb{Z}:\:|x|\leq R\sqrt{t}}\frac{\sigma_x}{\sqrt{t}}\]
Applying the strong law of large numbers again, one can readily check that the upper bound here converges to 0 for any $R<\infty$, $\mathbf{P}$-a.s. From (\ref{path}), we also deduce that, $\mathbf{P}$-a.s.,
\[\sup_{x\in\mathbb{Z}:\:|x|> R\sqrt{t}}P_\sigma\left(X_t=x\right)\leq P_\sigma\left(\tau(0,R\sqrt{t})<t\right)\rightarrow\mathbb{P}\left(\tau^B(0,R)<1\right),\]
where $\tau^B(0,R)$ is the exit time of $B$ from $(-R,R)$. Hence we find that
\[\limsup_{t\rightarrow\infty}\sup_{x\in\mathbb{Z}}P_\sigma\left(X_t=x\right)\leq \mathbb{P}\left(\tau^B(0,R)<1\right)\]
for any $R<\infty$, $\mathbf{P}$-a.s. Since this bound can be made arbitrarily small by taking $R$ large, we are done.

\section*{Acknowledgements}

This article was completed whilst the first author was a Visiting Associate Professor at Kyoto University, Research Institute for Mathematical Sciences. He would like to thank Takashi Kumagai and Ryoki Fukushima for their kind and generous hospitality.

\providecommand{\bysame}{\leavevmode\hbox to3em{\hrulefill}\thinspace}
\providecommand{\MR}{\relax\ifhmode\unskip\space\fi MR }
\providecommand{\MRhref}[2]{%
  \href{http://www.ams.org/mathscinet-getitem?mr=#1}{#2}
}
\providecommand{\href}[2]{#2}

\end{document}